\documentclass[11pt,reqno]{amsart}

\usepackage[T1]{fontenc}
\usepackage{amsmath,amssymb,amsthm,mathtools}
\usepackage{booktabs}
\usepackage{microtype}
\usepackage{xcolor}
\usepackage[colorlinks=true,linkcolor=blue,citecolor=blue,urlcolor=blue]{hyperref}
\usepackage[capitalize]{cleveref}

\usepackage{thmtools}
\declaretheorem[numberwithin=section]{theorem}
\declaretheorem[sibling=theorem]{proposition}
\declaretheorem[sibling=theorem]{lemma}
\declaretheorem[sibling=theorem]{corollary}
\declaretheorem[sibling=theorem,style=remark]{remark}

\newcommand{\R}{\mathbb{R}}
\newcommand{\N}{\mathbb{N}}
\newcommand{\Z}{\mathbb{Z}}
\newcommand{\PP}{\mathbb{P}}
\newcommand{\EE}{\mathbb{E}}
\newcommand{\abs}[1]{\left|#1\right|}
\newcommand{\norm}[1]{\left\|#1\right\|}
\newcommand{\Mel}{\mathcal{M}}
\newcommand{\eqd}{\stackrel{d}{=}}
\DeclareMathOperator{\Bin}{Bin}
\DeclareMathOperator{\Ber}{Bernoulli}
\DeclareMathOperator{\supp}{supp}

\title[Power-law and log-periodic degree tails in evolving networks]{Power-law and log-periodic
degree tails for a family of probability generating function equations arising in evolving networks}
\author[Qunqiang Feng]{Qunqiang Feng}
\address{School of Management, University of Science and Technology of China, Hefei, Anhui 230026, P.R.\ China}
\email{fengqq@ustc.edu.cn}
\author[Xiao-Ming Fu]{Xiao-Ming Fu}
\address{School of Mathematical Sciences, University of Science and Technology of China, Hefei, Anhui 230026, P.R.\ China}
\email{fuxm@ustc.edu.cn}
\author[Tianyang Sun]{Tianyang Sun}
\address{School of Mathematical Sciences, University of Science and Technology of China, Hefei, Anhui 230026, P.R.\ China}
\email{tysun@mail.ustc.edu.cn}
\date{July 2026}

\begin{document}

\begin{abstract}
For a fixed integer $j\ge1$ and $0<p<1$, we study the probability generating function (pgf) equation
\[
  (1+2p)\,g(x)=2p\,x^{j}+g\bigl(x-px+px^{2}\bigr),\qquad 0\le x\le1 ,
\]
which governs the limiting degree distribution $\{p_k\}$ of a family of evolving network models. 
The
cases $j=1$ and $j=2$ are the treelike fast-growth model of Feng and Hu and the homogeneous evolving
network of Feng, Li and Hu. We prove that for every $j$ the equation has a unique pgf solution, of
mean $2j$, and we determine its coefficient tail exactly:
\[
  p_k=k^{-1-\rho}\,\Psi_j(\log_\lambda k)+o\bigl(k^{-1-\rho}\bigr),
\]
where $\lambda=1+p$, $\rho=\log(1+2p)/\log(1+p)$ is independent of $j$, and $\Psi_j$ is continuous,
strictly positive and $1$-periodic, with explicit Fourier coefficients. 
This resolves two conjectures of Feng and coauthors: (1) the power-law order $p_k=\Theta(k^{-1-\rho})$ and (2) its refinement to the multiplicatively periodic form $p_k\sim\Psi_j(\log_\lambda k)\,k^{-1-\rho}$. 
The periodic factor is
genuinely non-constant for $p$ near $1$, and, for the two network models, for all $p$ outside a discrete set. 
Consequently, $p_k$ is asymptotic to no constant multiple of
$k^{-1-\rho}$. 
Our method is a
self-contained local analysis of the supercritical Galton-Watson process with offspring law
$1+\Ber(p)$, inspected at an independent geometric time. 
This time-changed process solves the equation observed by Feng and coauthors. The main results of this paper were obtained by
the multi-agent system Eureka and have subsequently been verified by the authors.
\end{abstract}

\maketitle

\section{Introduction}
\label{sec:intro}

Many large networks arising in nature, society and technology are \emph{scale-free}.
Examples include the Internet, the World Wide Web, and a range of citation, social, metabolic and protein-interaction networks. 
Their degree distributions $\{p_k\}$ obey a power law, $p_k$
decaying like $k^{-\gamma}$, with the exponent $\gamma$ empirically clustering in the range $(2,3)$
\cite{BA99,FFF99,AB02,Newman03,DM02,Hofstad17}. 
Since the work of Barab\'asi and Albert \cite{BA99}, this ubiquity has been explained through growing-network models in which power-law degrees emerge dynamically, for instance by preferential attachment. A substantial mathematical literature now
establishes heavy-tailed limiting degree distributions for such models. 
A basic question for any
given model is to pin down the \emph{precise} order of its degree tail: beyond knowing that the
tail is heavy, one wants the value of the exponent and any finer structure.

For many such models, the limiting degree law is specified only implicitly, through a
functional equation for its probability generating function (pgf), whose coefficients are the degree probabilities themselves and have no closed form. 
Extracting the tail is then a nontrivial analytic
problem, and the leading power law is often modulated by log-periodic oscillations. 
Two evolving network models of Feng and coauthors~\cite{FH17,FLH23}, whose degree pgfs satisfy such an equation, motivate this
paper. 

The \emph{homogeneous
evolving network} $\mathrm{EN}(n)$ of Feng, Li and Hu \cite{FLH23} is a randomized version of the
pseudofractal graph of Dorogovtsev, Goltsev and Mendes \cite{DGM}. It starts at time $0$ from a single edge. At each step, independently with probability $p\in(0,1)$, every edge of the current
graph recruits a new node joined to both of its endpoints. In the earlier treelike
``fast-growth'' preferential-attachment model of Feng and Hu \cite{FH17}, a tracked vertex instead
begins with a single edge. In both models the empirical degree distribution converges almost surely
to a law $\{p_k\}$ whose pgf satisfies a functional equation. Writing
$q=1-p$, these are the cases $j=2$ and $j=1$ of
\begin{equation}\label{eq:main}
  (1+2p)\,g(x)=2p\,x^{j}+g\bigl(x-px+px^{2}\bigr),\qquad 0\le x\le1,
\end{equation}
where $j\ge1$ is an integer and $x-px+px^{2}=px^{2}+qx$ (equation~(24) of \cite{FLH23} for $j=2$ and the case $j=1$ reads $(1+2p)g(x)=2px+g(px^{2}+qx)$). We study \eqref{eq:main} for every integer
$j\ge1$. As explained in \Cref{sec:rep}, general $j$ corresponds to a tracked vertex of initial
degree $j$. Its solution is the pgf
\begin{equation}\label{eq:pgf}
  g(x)=\sum_{k=j}^{\infty}p_k x^{k},\qquad p_k=\PP(X=k)>0,\quad \sum_{k\ge j}p_k=1,
\end{equation}
of an integer-valued random variable $X$ supported on $\{j,j+1,\dots\}$, with mean $\EE X=2j$.

The question raised in \cite{FH17,FLH23} is the behaviour of the tail as $k\to\infty$. On
the basis of simulations, Feng, Li and Hu
\cite{FLH23} conjectured a power-law order $p_k=\Theta\bigl(k^{-1-\rho}\bigr)$ with
\begin{equation}\label{eq:rhodef}
  \rho=\frac{\log(1+2p)}{\log(1+p)},
\end{equation}
their log-log degree plots oscillating between two parallel lines rather than following a single
straight line. Feng and Hu \cite{FH17} conjectured, more precisely, that
$p_k\sim\omega(k)\,k^{-1-\rho}$ for some continuous, strictly positive, multiplicatively periodic
$\omega$. Both were left open.

We prove both conjectures, for every $j$. Our main result (\Cref{thm:sharpmain}) is the sharp
asymptotic
\[
  p_k=k^{-1-\rho}\,\Psi_j(\log_\lambda k)+o\bigl(k^{-1-\rho}\bigr),\qquad \lambda=1+p,
\]
with $\Psi_j$ continuous, strictly positive and $1$-periodic, with explicitly computed Fourier
coefficients. As a continuous positive periodic function is bounded away from $0$ and $\infty$, this
contains the two-sided bound $p_k=\Theta(k^{-1-\rho})$ (\Cref{cor:theta}), proving the conjecture of
\cite{FLH23} at $j=2$ and extending it to the whole family --- the exponent $\rho$ being
\emph{independent of $j$}. Writing $\omega(k)=\Psi_j(\log_\lambda k)$ recovers the form conjectured
by \cite{FH17}. Moreover, $\Psi_j$ is genuinely non-constant for $p$ near $1$ and, for the two
network models, for all $p$ outside a discrete set (\Cref{cor:osc}); therefore, the tail is not
asymptotic to any constant multiple of $k^{-1-\rho}$. This is the log-periodic oscillation observed
numerically in \cite{FLH23}.

Our approach is probabilistic. As Feng and coauthors observed, \eqref{eq:main} is solved by a
supercritical Galton-Watson process with offspring law $1+\Ber(p)$, inspected at an independent
geometric time (\Cref{sec:rep}). 
The tail of $p_k$ is then extracted from the local probabilities
of this branching process by a uniform local limit theorem. The exponent
$\rho$ of \eqref{eq:rhodef} is the log-ratio of two rates governing the process: the geometric
mixing rate and the growth rate.
Functional equations of this type recur across duplication and
duplication-divergence random graphs (there, however, the degree-tail exponent is typically the root of a transcendental equation rather than the clean log-ratio~\eqref{eq:rhodef}); see Ispolatov,
Krapivsky and Yuryev \cite{IKY}, Hermann and Pfaffelhuber \cite{HP}, Jordan \cite{Jordan}, Jacquet,
Turowski and Szpankowski \cite{JTS}, and Turowski and Szpankowski \cite{TS}.

\Cref{sec:results} states the results precisely and outlines the proof strategy. \Cref{sec:rep} recalls the branching-process
representation. \Cref{sec:local} proves the local estimates, \Cref{sec:ground} the uniform local limit
theorem, and \Cref{sec:sharp} the sharp asymptotic --- hence
\Cref{thm:sharpmain,cor:theta}. \Cref{sec:const} analyses when $\Psi_j$ is genuinely oscillatory
and establishes \Cref{cor:osc}.

\section{Main results}
\label{sec:results}

Throughout, $j\ge1$ is a fixed integer and $0<p<1$. We write $q=1-p$, $\lambda=1+p$, $a=1+2p$, and
$\rho$ as in \eqref{eq:rhodef}, so that $a=\lambda^{\rho}$. Since $\lambda<a<\lambda^{2}$ (because
$a-\lambda=p>0$ and $\lambda^{2}-a=p^{2}>0$), one has $1<\rho<2$. In fact $\rho(p)$ is strictly
decreasing on $(0,1)$,%
\footnote{$\rho'(p)$ has the sign of $2(1+p)\log(1+p)-(1+2p)\log(1+2p)$, which is negative on
$(0,1)$: as $t\mapsto t\log t$ is convex and $(1+2p,1)$ majorizes $(1+p,1+p)$ (equal sums), Karamata's
inequality \cite{HLP} gives $(1+2p)\log(1+2p)\ge2(1+p)\log(1+p)$.}
with $\rho\to2$ as $p\to0^{+}$ and $\rho\to\log_{2}3$ as $p\to1^{-}$, so $\rho$ ranges exactly over
$(\log_{2}3,2)$ and the tail exponent $1+\rho$ over $(\log_{2}6,3)\subset(2,3)$, within the range
observed empirically for scale-free networks. Constants $C,c,\dots$ depend
only on $p$ and $j$ and may change between occurrences. $u_k=\Theta(v_k)$ means
$c\,v_k\le u_k\le C\,v_k$ for all large $k$.

The starting point, recalled from \cite{FH17,FLH23} in \Cref{sec:rep}, is that \eqref{eq:main} is
solved by the law of a geometrically time-changed branching process: if $(D^*_n)_{n\ge0}$ is the
Galton-Watson process with offspring law $1+\Ber(p)$ started from $j$ ancestors and $Y$ is an
independent geometric time, $\PP(Y=n)=\tfrac{2p}{a}a^{-n}$, then $X\eqd D^*_Y$, the equation has a
unique pgf solution, and
\begin{equation}\label{eq:mix}
  p_k=\sum_{n\ge0}\frac{2p}{a}\,a^{-n}\,\PP(D^*_n=k),\qquad k\ge j
\end{equation}
(\Cref{thm:rep}). Our results describe the tail of this mixture.

\begin{theorem}[Sharp asymptotics]\label{thm:sharpmain}
For every integer $j\ge1$ there is a continuous, strictly positive, $1$-periodic function $\Psi_j$,
with the explicit Fourier coefficients \eqref{eq:fourierintro}, such that
\begin{equation}\label{eq:sharpmain}
  p_k=k^{-1-\rho}\,\Psi_j(\log_\lambda k)+o\bigl(k^{-1-\rho}\bigr),\qquad k\to\infty .
\end{equation}
\end{theorem}

The Fourier coefficients announced in \Cref{thm:sharpmain} are explicit. Write $\chi:=2\pi/\log\lambda$
and let $\Mel_j(s):=\EE[V_j^{\,s-1}]$ be the Mellin transform \cite{FGD95} of the density of $V_j:=W_1+\cdots+W_j$,
the sum of $j$ independent copies of the martingale limit $W=\lim_n\lambda^{-n}D_n$ of the
single-ancestor branching process (\Cref{sec:rep}). Then the continuous, period-$1$ function
$\Psi_j$ has Fourier coefficients
\begin{equation}\label{eq:fourierintro}
  \widehat\Psi_j(m):=\int_0^1\Psi_j(u)e^{2\pi imu}\,du=\frac{2p}{a\log\lambda}\,\Mel_j\bigl(\rho+1+i\chi m\bigr),
  \qquad m\in\Z,
\end{equation}
so in particular the mean value $\widehat\Psi_j(0)=\frac{2p}{a\log\lambda}\,\EE[V_j^{\,\rho}]$ is
strictly positive. The function $\Psi_j$ itself and \eqref{eq:fourierintro} are constructed in
\Cref{sec:sharp} (\Cref{prop:phi,prop:fourier}).

\begin{corollary}[Power-law order]\label{cor:theta}
For every integer $j\ge1$, $p_k=\Theta\bigl(k^{-1-\rho}\bigr)$. 
In particular, the limiting degree
distributions of the homogeneous evolving network of \cite{FLH23} ($j=2$) and of the treelike
fast-growth model of \cite{FH17} ($j=1$) satisfy $p_k=\Theta(k^{-1-\rho})$. The case $j=2$ proves
the power-law conjecture of \cite{FLH23}, and the exponent does not depend on $j$.
\end{corollary}

\begin{proof}
$\Psi_j$ is continuous, positive and $1$-periodic, thus
$0<\min\Psi_j\le\Psi_j\le\max\Psi_j<\infty$. Then, \eqref{eq:sharpmain} gives
$c\,k^{-1-\rho}\le p_k\le C\,k^{-1-\rho}$ for all large $k$.
\end{proof}

Writing $\omega(k)=\Psi_j(\log_\lambda k)$, \eqref{eq:sharpmain} is exactly the multiplicatively
periodic refinement $p_k\sim\omega(k)\,k^{-1-\rho}$ conjectured by \cite{FH17}. Except on a discrete
set of $p$, we prove below that $\Psi_j$ is non-constant and hence that the oscillation is
\emph{genuine}.

\begin{corollary}[Genuine oscillation]\label{cor:osc}
For $j\in\{1,2\}$ there is a discrete set $\mathcal Z_1\subset(0,1)$ such that $\Psi_j$ is
non-constant for every $p\in(0,1)\setminus\mathcal Z_1$. Moreover, for every $j\ge1$, $\Psi_j$
is non-constant for all $p$ in a neighbourhood of $1$. In either case $p_k$ is then not asymptotic to
any constant multiple of $k^{-1-\rho}$.
\end{corollary}

\Cref{cor:osc} at $j=1$ establishes the conjecture of \cite{FH17} in its precise (non-constant)
form, and at $j=2$ the oscillation seen numerically in \cite{FLH23}.

\paragraph{Strategy of proof.} By \eqref{eq:mix} the tail of $p_k$ is governed by the local
probabilities $\PP(D^*_n=k)$, weighted geometrically, and the exponent $\rho=\log a/\log\lambda$
balances the two competing rates: a generation with $D^*_n\approx k$ occurs near
$n\approx\log_\lambda k$, where the geometric weight $a^{-n}\approx k^{-\rho}$ meets the local
probability $\PP(D^*_n=k)\approx\lambda^{-n}\approx k^{-1}$, producing the order $k^{-1-\rho}$;
equivalently, in the singularity analysis of Flajolet and Odlyzko \cite{FO} the substitution
$x\mapsto x-px+px^{2}$ multiplies $1-x$ by $\lambda$ to first order while \eqref{eq:main} carries
the factor $a$, so a singular part $(1-x)^{\rho}$ is reproduced exactly when $a=\lambda^{\rho}$. To
upgrade this heuristic to the sharp asymptotic \eqref{eq:sharpmain} we analyse the local
probabilities directly (\Cref{sec:local,sec:ground}): a global envelope, uniform exponential
moments, a right-tail local bound, and a uniform local limit theorem obtained from the
$L^{1}$-decay of the characteristic function, the only external input beyond standard
branching-process and Fourier theory being the classical strict positivity of the martingale-limit density. Periodizing the resulting density profile against the
geometric weights yields $\Psi_j$ and its Fourier coefficients (\Cref{sec:sharp}); the period-$1$
structure in $\log_\lambda k$ is the shadow of the discrete scaling $\lambda^{n}$
(\Cref{rem:koenigs}).

\paragraph{Correspondence with the notation of \cite{FH17,FLH23}.} We follow Feng and coauthors as
closely as the enlarged setting permits, and, where the two papers differ, we adopt the later one
\cite{FLH23}. The offspring generating function $f(x)=px^{2}+qx$ and its $n$-th iterate $f^{(n)}$
are as in \cite[eqs.~(22),(23)]{FLH23}; the degree-process generating function $f_n=[f^{(n)}]^{j}$
is their $f_m$ (there $j=2$). Our $j$-ancestor degree process $D^*_n$ (\Cref{sec:rep}) is their
$D^*_m$, with initial degree $j$ in place of $2$ (for $j=1$ it is the process $D_o$ of \cite{FH17}),
and $D_n$ is its single-ancestor building block. The geometric mixing time $Y$ is that of
\cite{FLH23} (denoted $Z$ in \cite{FH17}), and $X\eqd D^*_Y$. The tail exponent $1+\rho$ is
denoted $\gamma$ in \cite{FLH23} (and $\tau$ in \cite{FH17}); we keep $1+\rho$ throughout, reserving
the symbol $\gamma$ for the unrelated phase-decay rate $q+p\delta_0$ of \Cref{lem:envelope}. We
write the generation index
as $n$ (the age $m$ of \cite{FLH23}). The martingale limit
$W=\lim_n\lambda^{-n}D_n$ coincides in law with the limit $M$ of the uniform model of \cite{FH17}
($\EE W=1$, $\operatorname{Var}W=q/\lambda$); it is \emph{not} the edge-count martingale
$Z=\lim_n E_n/(1+2p)^n$ of \cite{FLH23}, which is a different normalization.

\section{The branching-process representation}
\label{sec:rep}

The results of \Cref{sec:results} rest on the geometric-mixture representation \eqref{eq:mix} of
the solution $g$, taken there as a starting point. This section makes that representation precise
and proves it. We introduce the Galton-Watson process underlying the mixture, establish the
representation together with the uniqueness of the pgf solution and the mean $\EE X=2j$
(\Cref{thm:rep}), and record three structural facts about the process: contiguous support, span
$1$, and finite exponential moments. The local analysis of the following sections depends on these
facts.

Let $\xi=1+\Ber(p)$ be the offspring random variable, taking the value $1$ with probability
$q=1-p$ and the value $2$ with probability $p$, with generating function
\begin{equation}\label{eq:offspring}
  f(x)=\EE x^{\xi}=px^{2}+qx=x(q+px),\qquad f'(1)=q+2p=1+p=\lambda .
\end{equation}
Note that $f$ coincides with the composition appearing in \eqref{eq:main}:
\begin{equation}\label{eq:Ff}
  f(x)=px^{2}+qx=x-px+px^{2}.
\end{equation}
Let $(D_n)_{n\ge0}$ be the Galton-Watson process with offspring law $\xi$ started from a single
ancestor, $D_0=1$, and let $(D^*_n)_{n\ge0}$ be the same process started from $j$ ancestors,
realized as
\begin{equation}\label{eq:Ztwo}
  D^*_n=D_n^{(1)}+\cdots+D_n^{(j)},\qquad D^*_0=j,
\end{equation}
where $D^{(1)},\dots,D^{(j)}$ are independent copies of $D$.

\emph{Network interpretation.} In the models of \cite{FH17,FLH23} the degree of a fixed vertex
evolves by exactly this law: at each step every incident edge independently survives and, with
probability $p$, additionally spawns one new incident edge through the common neighbour recruited
by that edge, so it is replaced by $\xi=1+\Ber(p)$ edges. A vertex present with degree $j$ therefore
has degree process $D^*_n$; a newly recruited vertex has degree $j=2$ in \cite{FLH23} and $j=1$ in
the treelike model of \cite{FH17}, and inspecting its degree at the independent geometric age $Y$ of
a uniformly sampled vertex (\Cref{thm:rep}) gives $X\eqd D^*_Y$. The per-vertex degree rate
$\lambda=1+p$ and the graph-growth rate $a=1+2p$ satisfy $\rho=\log a/\log\lambda$. They differ
because each recruited node adds two edges to the graph but only one to each endpoint.

Writing $f^{(n)}$ for the $n$-fold
iterate of $f$ (with $f^{(0)}(x)=x$), the branching property gives $\EE x^{D_n}=f^{(n)}(x)$ and hence
\begin{equation}\label{eq:Gn}
  f_n(x):=\EE x^{D^*_n}=f^{(n)}(x)^{j},\qquad f_0(x)=x^{j},\qquad
  f_{n+1}(x)=f_n\bigl(f(x)\bigr).
\end{equation}

The next statement gives the geometric-mixture representation of $g$, the uniqueness of the
pgf solution, and the mean. For $j=2$ it is due to Feng, Li and Hu \cite{FLH23}, and for $j=1$ to
Feng and Hu \cite{FH17}; the general $j$ is identical. We record it in the present notation, since
the process $(D^*_n)$ and the mixture
\eqref{eq:mix} are used throughout, and give only a sketch.

\begin{theorem}[Representation, uniqueness, and mean; \cite{FH17,FLH23}]\label{thm:rep}
Let $Y$ be a geometric random variable, independent of $(D^*_n)$, with
\begin{equation}\label{eq:geom}
  \PP(Y=n)=\frac{2p}{a}\,a^{-n},\qquad n=0,1,2,\dots
\end{equation}
Then $g(x)=\EE x^{D^*_Y}=\sum_{n\ge0}\tfrac{2p}{a}a^{-n}f_n(x)$ is a pgf supported on
$\{j,j+1,\dots\}$, and it is the unique pgf solution of \eqref{eq:main}. Consequently the mixture
\eqref{eq:mix} holds, and $\EE X=2j$.
\end{theorem}

\begin{proof}[Sketch; see \cite{FH17,FLH23} for details]
The weights \eqref{eq:geom} are nonnegative and sum to $1$ (since $a-1=2p$), so
$g=\sum_n\PP(Y=n)f_n$ is a convex combination of pgfs, hence the
generating function of $D^*_Y$; as $D^*_n\ge j$, it is supported on $\{j,j+1,\dots\}$. That $g$ solves
\eqref{eq:main} follows from $f_{n+1}=f_n\circ f$ in \eqref{eq:Gn} and reindexing, which give
$g(f(x))=a\,g(x)-2p\,x^{j}$; by \eqref{eq:Ff} this is \eqref{eq:main}. For uniqueness, any power
series solution $h(x)=\sum_k h_kx^{k}$ satisfies $h(f(x))=a\,h(x)-2p\,x^{j}$. Write
$[x^{k}]F$ for the coefficient of $x^{k}$ in a power series $F$, and note that
$f(x)^{i}=x^{i}(q+px)^{i}$ has lowest-order term $q^{i}x^{i}$. Comparing coefficients of $x^{k}$ gives
\begin{equation}\label{eq:recursion}
  \bigl(a-q^{k}\bigr)h_k=2p\,\mathbf 1_{\{k=j\}}+\sum_{i<k}h_i\,[x^{k}]f(x)^{i};
\end{equation}
since $a-q^{k}=1+2p-(1-p)^{k}>0$, this determines $h_k$ from $h_0,\dots,h_{k-1}$ (the first $j$
coefficients $h_0=\cdots=h_{j-1}=0$ being forced), so the solution is unique. Finally
$\EE D_n=\lambda^{n}$ gives $\EE D^*_n=j\lambda^{n}$, so by \eqref{eq:mix} and $a-\lambda=p$,
\[
  \EE X=\sum_{n\ge0}\frac{2p}{a}a^{-n}\,j\lambda^{n}
  =\frac{2pj}{a}\cdot\frac{1}{1-\lambda/a}
  =\frac{2pj}{a-\lambda}=2j ,
\]
equivalently $g'(1)=2j$.
\end{proof}

From now on we work exclusively with the process $(D^*_n)$ and the mixture \eqref{eq:mix}. We
record the two structural facts used repeatedly: the single-ancestor process $D_n$ is supported
on $\{1,2,\dots,2^{n}\}$ (by induction: given $D_n=y$, $D_{n+1}=y+\Bin(y,p)$ has support
$\{y,\dots,2y\}$, and $\bigcup_{y=1}^{2^{n}}\{y,\dots,2y\}=\{1,\dots,2^{n+1}\}$), so, being the
sum of $j$ independent copies,
\begin{equation}\label{eq:support}
  \supp(D^*_n)=\{j,j+1,\dots,j\,2^{n}\},
\end{equation}
a contiguous set, whence $(D^*_n)$ has span $1$; and the offspring is bounded, so all exponential
moments are finite.

\section{Local estimates for the branching process}
\label{sec:local}

We first bound the single-ancestor local probabilities uniformly, then establish uniform
exponential moments on the scale $\lambda^{n}$, and finally combine the two into a right-tail
local bound for $D^*_n$ that is summable against the mixture weights.

\subsection{A global envelope for the local probabilities}

\begin{lemma}\label{lem:envelope}
There is a constant $C<\infty$, depending only on $p$, such that
\[
  \sup_{l\ge1}\PP(D_n=l)\le C\lambda^{-n}\qquad\text{for all }n\ge0 .
\]
\end{lemma}

\begin{proof}
Since $\PP(D_n=l)=\frac{1}{2\pi}\int_{-\pi}^{\pi}f^{(n)}(e^{it})e^{-i l t}\,dt$, it suffices to prove
\begin{equation}\label{eq:I-goal}
  I_n:=\int_{-\pi}^{\pi}\abs{f^{(n)}(e^{it})}\,dt\ \le\ C\lambda^{-n}.
\end{equation}

\emph{Product form and two factor bounds.} Write $s_r=f^{(r)}(e^{it})$. As $f^{(r+1)}=f\circ f^{(r)}$ and
$f(s)=s(q+ps)$, we have $\abs{f^{(r+1)}}=\abs{f^{(r)}}\,\abs{q+ps_r}$; hence, with $\abs{f^{(0)}}=1$,
\begin{equation}\label{eq:prod}
  \abs{f^{(n)}(e^{it})}=\prod_{r=0}^{n-1}\abs{q+ps_r},
\end{equation}
and $\abs{s_r}=\abs{f^{(r)}(e^{it})}$ is nonincreasing in $r$. From
$1-\abs{q+ps}^2=2pq(1-\operatorname{Re}s)+p^2(1-\abs{s}^2)\ge 2pq(1-\operatorname{Re}s)$ for $\abs{s}\le1$,
and $\sqrt{1-x}\le e^{-x/2}$ (valid as $2pq(1-\operatorname{Re}s)\le 4pq\le1$),
\begin{equation}\label{eq:fac1}
  \abs{q+ps_r}\ \le\ \exp\!\big(-pq\,\Delta_r(t)\big),\qquad
  \Delta_r(t):=1-\operatorname{Re}s_r=\EE\big(1-\cos(tD_r)\big)\ge0 .
\end{equation}
Put $\delta_0:=\dfrac{p}{2(1+p)}$ and $\gamma:=q+p\delta_0$. Then $\gamma<\lambda^{-1}$, since
$\gamma<\tfrac1{1+p}\iff q(1+p)+\tfrac{p^2}2<1\iff 1-\tfrac{p^2}2<1$. Consequently
\begin{equation}\label{eq:fac2}
  \abs{s_r}\le\delta_0\ \Longrightarrow\ \abs{q+ps_r}\le\gamma<\lambda^{-1}.
\end{equation}

\emph{A dispersion estimate.} Set $m(t):=\max\{0,\lceil\log_\lambda(1/\abs{t})\rceil\}$ for
$0<\abs{t}\le\pi$. We claim there is $c_*>0$ (depending only on $p$) with
\begin{equation}\label{eq:disp}
  \Delta_\nu(t)\ \ge\ c_B:=1-e^{-pq c_*}\qquad\text{for all }\nu\ge m(t).
\end{equation}
First, using $1-\cos x\ge\frac{x^2}{2}-\frac{x^4}{24}$, $\EE D_r^2\ge(\EE D_r)^2=\lambda^{2r}$, and
$\EE D_r^4\le C_4\lambda^{4r}$ (conditioning on $D_r$ and expanding $D_{r+1}=\sum_{i\le D_r}\xi_i$ with
bounded $\xi_i$ gives $\EE[D_{r+1}^k\mid D_r]\le\lambda^{k}D_r^{k}+C_k D_r^{k-1}$, so an induction on
$r$ yields $\EE D_r^k\le C_k\lambda^{kr}$ for every fixed $k$), there is $\eta\in(0,1)$ with
\begin{equation}\label{eq:small}
  \Delta_r(t)\ge\tfrac14 t^2\lambda^{2r}\qquad\text{whenever }\abs{t}\lambda^r\le\eta.
\end{equation}
Set $c_*:=\min\{\tfrac14(\eta/\lambda)^2,\,1-\cos\eta,\,1-\cos1\}>0$, so that
$c_B=1-e^{-pqc_*}\le pqc_*\le c_*$. We first show $\sum_{k<\nu}\Delta_k(t)\ge c_*$ for every
$\nu\ge\max\{m(t),1\}$. If $\abs{t}\le\eta$, let $k_1$ be the largest $k$ with
$\abs{t}\lambda^{k}\le\eta$; then $k_1<\log_\lambda(1/\abs{t})\le\nu$ and $\abs{t}\lambda^{k_1}>\eta/\lambda$,
so by \eqref{eq:small}
$\sum_{k<\nu}\Delta_k(t)\ge\sum_{k\le k_1}\tfrac14 t^2\lambda^{2k}\ge\tfrac14 t^2\lambda^{2k_1}\ge\tfrac14(\eta/\lambda)^2\ge c_*$.
If $\abs{t}>\eta$, then, since $\nu\ge1$, $\sum_{k<\nu}\Delta_k(t)\ge \Delta_0(t)=1-\cos t\ge1-\cos\eta\ge c_*$
(as $x\mapsto1-\cos x$ increases on $[0,\pi]$). For such $\nu$, \eqref{eq:prod}--\eqref{eq:fac1} then give
$\abs{s_\nu}=\abs{f^{(\nu)}(e^{it})}\le e^{-pq\sum_{k<\nu}\Delta_k}\le e^{-pqc_*}$, whence
$\Delta_\nu(t)\ge1-\abs{s_\nu}\ge c_B$. The only case with $\nu\ge m(t)$ left uncovered is $\nu=m(t)=0$,
which forces $\abs{t}\ge1$ and gives directly $\Delta_0(t)=1-\cos t\ge1-\cos1\ge c_B$. This proves
\eqref{eq:disp}; in particular $\sum_{r=0}^{R-1}\Delta_r(t)\ge c_B\,(R-m(t))_+$.

\emph{Pointwise decay.} Let $r_0(t)$ be the first $r$ with $\abs{s_r}\le\delta_0$. By
\eqref{eq:fac1} and the last display, $\abs{s_r}\le e^{-pqc_B(r-m(t))_+}$, so $r_0(t)\le m(t)+K_0$
with $K_0:=\lceil\log(1/\delta_0)/(pqc_B)\rceil$. Because $\abs{s_r}$ is nonincreasing,
\eqref{eq:fac2} bounds each factor with $r\ge r_0(t)$ by $\gamma$, while the factors with
$r<r_0(t)$ are $\le1$; hence by \eqref{eq:prod},
\begin{equation}\label{eq:ptwise}
  \abs{f^{(n)}(e^{it})}\ \le\ \gamma^{\,(n-r_0(t))_+}\ \le\ C_0\,\gamma^{\,(n-m(t))_+},\qquad
  C_0:=\gamma^{-K_0}.
\end{equation}

\emph{Integration.} Partition $[-\pi,\pi]$, up to a null set of endpoints, into
$B_0=\{1<\abs{t}\le\pi\}$ and, for $m\ge1$, $B_m=\{\lambda^{-m}<\abs{t}\le\lambda^{-m+1}\}$; then
$\abs{B_m}=O(\lambda^{-m})$ and $m(t)=m$ on the interior of $B_m$. By \eqref{eq:ptwise},
\[
  I_n\ \le\ C\sum_{m\ge0}\lambda^{-m}\,\gamma^{(n-m)_+}
  \ =\ C\Big(\lambda^{-n}\sum_{m=0}^{n}(\lambda\gamma)^{\,n-m}+\sum_{m>n}\lambda^{-m}\Big).
\]
Since $\lambda\gamma=1-\tfrac{p^2}{2}<1$, the first sum is $\le\lambda^{-n}/(1-\lambda\gamma)=O(\lambda^{-n})$
and the second is $O(\lambda^{-n})$. This proves \eqref{eq:I-goal}, and the lemma follows.
\end{proof}

\subsection{Uniform exponential moments}

\begin{lemma}\label{lem:expmom}
There exist $\theta>0$ and $M<\infty$, depending only on $p$, such that
\[
  \sup_{n\ge0}\EE\exp\!\Bigl\{\theta\,\frac{D_n}{\lambda^{n}}\Bigr\}\le M,
  \qquad
  \sup_{n\ge0}\EE\exp\!\Bigl\{\theta\,\frac{D^*_n}{\lambda^{n}}\Bigr\}\le M^{j}.
\]
Consequently there are constants $C,c>0$ with
\[
  \PP(D_n\ge x\lambda^{n})\le C e^{-cx},\qquad
  \PP(D^*_n\ge x\lambda^{n})\le C e^{-cx},\qquad n\ge0,\ x\ge0 .
\]
\end{lemma}

\begin{proof}
Let $\varphi(u)=\EE e^{u\xi}=qe^{u}+pe^{2u}$. Then $\varphi(0)=1$ and $\varphi'(0)=\lambda$, and
$\log\varphi$ is smooth near $0$, so there are $B,T>0$ with
\begin{equation}\label{eq:logphi}
  \log\varphi(u)\le\lambda u+Bu^{2},\qquad 0\le u\le T .
\end{equation}
Write $A_n(t)=\EE\exp\{tD_n/\lambda^{n}\}$. Conditioning on $D_n$ and using the branching
property, then applying \eqref{eq:logphi} with $u=t/\lambda^{n+1}$ (valid once $t/\lambda^{n+1}\le T$), we obtain
\[
  A_{n+1}(t)=\EE\Bigl[\varphi\bigl(t/\lambda^{n+1}\bigr)^{D_n}\Bigr]
  \le\EE\exp\!\Bigl\{\Bigl(\frac{t}{\lambda^{n}}+\frac{Bt^{2}}{\lambda^{2n+2}}\Bigr)D_n\Bigr\}
  =A_n\!\Bigl(t+\frac{Bt^{2}}{\lambda^{n+2}}\Bigr).
\]
Iterating this downward, $A_n(\theta)\le A_{n-1}(\theta_1)\le\cdots\le A_0(\theta_n)$, where
$\theta_0=\theta$ and $\theta_{i+1}=\theta_i+B\theta_i^{2}\lambda^{-(n-i)-1}$. If all
$\theta_i\le2\theta$, then the total increment is at most
$\sum_{r\ge2}B(2\theta)^{2}\lambda^{-r}=4B\theta^{2}S$ with $S:=\sum_{r\ge2}\lambda^{-r}<\infty$;
hence, provided
\[
  \theta+4B\theta^{2}S\le2\theta\quad\text{and}\quad 2\theta\le T,
  \qquad\text{i.e.}\qquad \theta\le\min\Bigl\{\tfrac{1}{4BS},\tfrac{T}{2}\Bigr\},
\]
a straightforward induction shows $\theta_i\le2\theta$ for all $i$. Fix such a $\theta$. Then
$A_n(\theta)\le A_0(2\theta)=\EE e^{2\theta D_0}=e^{2\theta}=:M$, uniformly in $n$. For $D^*_n$,
independence of the $j$ ancestors gives
$\EE e^{\theta D^*_n/\lambda^{n}}=A_n(\theta)^{j}\le M^{j}$. The tail bounds follow from Markov's
inequality: $\PP(D_n\ge x\lambda^{n})\le M e^{-\theta x}$, and likewise for $D^*_n$ with $M^{j}$;
relabel constants.
\end{proof}

\subsection{A right-tail local bound}

We first bound the single-ancestor local probabilities in the right tail; the bound for $D^*_n$
(for each fixed $j$, uniformly in $n$ and $k$) then follows by a union bound.

\begin{lemma}\label{lem:RTsingle}
There exist $C,c>0$, depending only on $p$, such that
\[
  \PP(D_n=l)\le C\lambda^{-n}\exp\!\Bigl\{-c\,\frac{l}{\lambda^{n}}\Bigr\},\qquad n\ge0,\ l\ge1 .
\]
\end{lemma}

\begin{proof}
Write $u_n(l)=\PP(D_n=l)$. Splitting on the first generation ($\xi=1$ with probability $q$,
$\xi=2$ with probability $p$, each child founding an independent copy of $D_{n-1}$) yields
\[
  u_n(l)=q\,u_{n-1}(l)+p\,(u_{n-1}*u_{n-1})(l),\qquad n\ge1 .
\]
In the convolution at least one summand is $\ge l/2$, so the envelope of \Cref{lem:envelope}
and the exponential tail of \Cref{lem:expmom} give
\[
  (u_{n-1}*u_{n-1})(l)\le 2\bigl(\sup_m u_{n-1}(m)\bigr)\PP(D_{n-1}\ge l/2)
  \le C_0\lambda^{-(n-1)}e^{-c_0 l/\lambda^{n-1}} .
\]
Unrolling from $n$ down to $0$, with $u_0(l)=\mathbf 1_{\{l=1\}}$, gives
\[
  u_n(l)\le q^{n}\mathbf 1_{\{l=1\}}
  +C_0 p\sum_{s=0}^{n-1}q^{\,n-1-s}\lambda^{-s}e^{-c_0 l/\lambda^{s}} .
\]
Put $y=l/\lambda^{n}$ and reindex by $u=n-s\ge1$: the sum equals
\[
  \lambda^{-n}\sum_{u=1}^{n}q^{\,u-1}\lambda^{u}e^{-c_0\lambda^{u}y}
  \ \le\ \lambda^{-n}e^{-c_0\lambda y}\sum_{u\ge1}q^{\,u-1}\lambda^{u}
  \ =\ \frac{\lambda}{1-q\lambda}\,\lambda^{-n}e^{-c_0\lambda y},
\]
using $\lambda^{u}\ge\lambda$ for $u\ge1$ and $q\lambda=1-p^{2}<1$. Finally
$q^{n}\mathbf 1_{\{l=1\}}\le\lambda^{-n}\le C\lambda^{-n}e^{-c l/\lambda^{n}}$ at $l=1$ (as
$q\lambda=1-p^{2}<1$ gives $q<\lambda^{-1}$, and $e^{-c/\lambda^{n}}\ge e^{-c}$ since $\lambda^{-n}\le1$).
Both terms are thus $\le C\lambda^{-n}e^{-c l/\lambda^{n}}$ with $c=c_0\lambda$, after relabelling.
\end{proof}

\begin{lemma}\label{lem:RT}
There exist $C,c>0$, depending only on $p$ and $j$, such that
\begin{equation}\label{eq:RT}
  \PP(D^*_n=k)\le C\lambda^{-n}\exp\!\Bigl\{-c\,\frac{k}{\lambda^{n}}\Bigr\},
  \qquad n\ge0,\ k\ge j .
\end{equation}
Consequently, writing $N_k=\lfloor\log_\lambda k\rfloor$ and $m=n-N_k$,
\begin{equation}\label{eq:Dm}
  \PP(D^*_n=k)\le\lambda^{-n}\mathcal D_m,\qquad
  \mathcal D_m=\begin{cases}C\exp\{-c\lambda^{-m}\},& m<0,\\[2pt] C,& m\ge0,\end{cases}
\end{equation}
and $\mathcal D_m$ is weighted-summable:
\begin{equation}\label{eq:Dsum}
  \sum_{m\in\Z}\mathcal D_m\,\lambda^{-(\rho+1)m}<\infty .
\end{equation}
\end{lemma}

\begin{proof}
Write $u_n(l)=\PP(D_n=l)$; by \eqref{eq:Ztwo}, $\PP(D^*_n=k)$ is the value at $k$ of the $j$-fold
convolution of $u_n$. In every nonzero term $k=l_1+\cdots+l_j$ some summand is $\ge k/j$, so a
union bound over which summand is largest gives
\[
  \PP(D^*_n=k)\le j\sum_{l\ge k/j}u_n(l)\,v_n(k-l),
\]
where $v_n$ is the $(j-1)$-fold convolution of $u_n$ (with $v_n=\delta_0$ if $j=1$). By
\Cref{lem:RTsingle}, $u_n(l)\le C\lambda^{-n}e^{-c(k/j)/\lambda^{n}}$ for $l\ge k/j$, while
$\sum_{l}v_n(k-l)\le\sum_m v_n(m)=1$; hence $\PP(D^*_n=k)\le jC\lambda^{-n}e^{-c(k/j)/\lambda^{n}}$,
which is \eqref{eq:RT} after absorbing $j$ into $C$ and relabelling $c$.

For \eqref{eq:Dm}, note $\lambda^{N_k}\le k<\lambda^{N_k+1}$, so
$k/\lambda^{n}=(k/\lambda^{N_k})\lambda^{-m}\ge\lambda^{-m}$; when $m<0$ this and \eqref{eq:RT}
give $\PP(D^*_n=k)\le C\lambda^{-n}e^{-c\lambda^{-m}}$, while for $m\ge0$ we use \eqref{eq:RT}
with the exponential $\le1$. Finally,
\[
  \sum_{m\ge0}C\lambda^{-(\rho+1)m}<\infty,
  \qquad
  \sum_{m<0}C e^{-c\lambda^{-m}}\lambda^{-(\rho+1)m}
  =C\sum_{\ell\ge1}e^{-c\lambda^{\ell}}\lambda^{(\rho+1)\ell}<\infty,
\]
the second series converging because $e^{-c\lambda^{\ell}}$ decays doubly exponentially in
$\ell$. This proves \eqref{eq:Dsum}.
\end{proof}

\section{Analytic groundwork for the sharp asymptotic}\label{sec:ground}

The proof of \Cref{thm:sharpmain} occupies this section and the next; we keep all notation
of \Cref{sec:rep,sec:local}. In particular $\gamma=q+p\delta_0<\lambda^{-1}$,
$\delta_0=\tfrac{p}{2(1+p)}$, $C_0=\gamma^{-K_0}$ and
$m(t)=\max\{0,\lceil\log_\lambda(1/\abs t)\rceil\}$ are as in \Cref{lem:envelope}, and we set
\begin{equation}\label{eq:kappa}
  \kappa:=-\log_\lambda\gamma>1 .
\end{equation}
Three inputs from the preceding sections are used repeatedly: the pointwise envelope
\eqref{eq:ptwise}, $\abs{f^{(n)}(e^{it})}\le C_0\gamma^{(n-m(t))_+}$; the right-tail local bound
\eqref{eq:RT}; and the classical fact that the martingale limit $W=\lim_n\lambda^{-n}D_n$
($\EE W=1$ by the Kesten-Stigum theorem \cite{KS66}) has a continuous density $w$, strictly positive on $(0,\infty)$ --- the process being
supercritical, non-extinct ($p_0=0$) and of Schr\"oder type ($p_1=q>0$); see Athreya and Ney
\cite{AthreyaNey}, Dubuc and Seneta \cite{DubucSeneta}, and Fleischmann and Wachtel \cite{FW}.
(Continuity is in any case reproved in \Cref{lem:cf}.)

Since $D^*_n=D_n^{(1)}+\cdots+D_n^{(j)}$, we have $D^*_n/\lambda^n\to V_j:=W_1+\cdots+W_j$ almost surely,
with $W_1,\dots,W_j$ independent copies of $W$; write $w_j$ for the density of $V_j$, i.e.\ the
$j$-fold convolution $w^{*j}$. Write $\psi_n(\theta):=f^{(n)}(e^{i\theta/\lambda^n})=\EE e^{i\theta D_n/\lambda^n}$
for $\abs\theta\le\pi\lambda^n$ and $\psi_n(\theta):=0$ otherwise, so that $\psi_n(\theta)^j$
agrees with the characteristic function of $D^*_n/\lambda^n$ on $\abs\theta\le\pi\lambda^n$, the only
range entering the inversion below.

\begin{lemma}[Characteristic functions, densities, and decay]\label{lem:cf}
The characteristic function $\widehat w(\theta):=\EE e^{i\theta W}$ and each $\psi_n$ satisfy
\begin{equation}\label{eq:cfbound}
  \abs{\widehat w(\theta)}\le C_1(1+\abs\theta)^{-\kappa},\qquad
  \abs{\psi_n(\theta)}\le C_1(1+\abs\theta)^{-\kappa}\quad(\text{uniformly in }n),
\end{equation}
with $C_1$ depending only on $p$. Consequently $\widehat{w_j}=\widehat w^{\,j}$ satisfies
$\abs{\widehat{w_j}(\theta)}\le C_1^{j}(1+\abs\theta)^{-j\kappa}$ with $j\kappa>1$, so
$\widehat{w_j}\in L^1(\R)$ and $V_j$ has the bounded continuous density
$w_j(x)=\frac1{2\pi}\int_\R\widehat w(\theta)^{j}e^{-i\theta x}\,d\theta$. Moreover $w_j$ is
supported on $[0,\infty)$, is strictly positive on $(0,\infty)$, has $\norm{w_j}_\infty<\infty$
(so $x^{\rho+1}w_j(x)=O(x^{\rho+1})$ as $x\downarrow0$), and satisfies $w_j(x)\le C_2 e^{-c_2 x}$
for $x\ge1$.
\end{lemma}

\begin{proof}
\emph{Decay of $\psi_n$.} If $1\le\abs\theta\le\lambda^{n}$, put $t=\theta/\lambda^{n}$, so
$\abs t\le1$ and $m(t)=\lceil n-\log_\lambda\abs\theta\rceil\le n-\log_\lambda\abs\theta+1$, whence
$(n-m(t))_+\ge\log_\lambda\abs\theta-1$ and \eqref{eq:ptwise} gives
$\abs{\psi_n(\theta)}\le C_0\gamma^{\log_\lambda\abs\theta-1}=C_0\gamma^{-1}\abs\theta^{-\kappa}$. If
$\lambda^{n}<\abs\theta\le\pi\lambda^{n}$, then $\abs t\in(1,\pi]$, $m(t)=0$, so
$\abs{\psi_n(\theta)}\le C_0\gamma^{n}=C_0\lambda^{-\kappa n}\le C_0\pi^{\kappa}\abs\theta^{-\kappa}$.
For $\abs\theta<1$, $\abs{\psi_n}\le1$; for $\abs\theta>\pi\lambda^{n}$, $\psi_n=0$. This is the
uniform bound on $\psi_n$; since $D_n/\lambda^{n}\to W$ a.s., bounded convergence gives
$\psi_n(\theta)\to\widehat w(\theta)$ pointwise, so \eqref{eq:cfbound} holds for $\widehat w$ too.

\emph{Densities.} As $\kappa>1$, $\widehat w\in L^1$, so Fourier inversion gives $W$ the bounded
continuous density $w$; $w=0$ on $(-\infty,0)$ since $W>0$, and $w>0$ on $(0,\infty)$ as recalled
above. Since $V_j$ has characteristic function $\widehat w^{\,j}$ with $j\kappa>1$, the same
inversion gives the bounded continuous density $w_j=w^{*j}$; being a convolution of nonnegative
functions supported on $[0,\infty)$ and positive on $(0,\infty)$, $w_j$ is supported on
$[0,\infty)$ and positive on $(0,\infty)$. As $w_j$ is bounded,
$x^{\rho+1}w_j(x)\le\norm{w_j}_\infty x^{\rho+1}\to0$ as $x\downarrow0$.

\emph{Exponential tail.} By \Cref{lem:llt} below, for $x>0$ and integers $k_n$ with
$k_n/\lambda^{n}\to x$ one has $w_j(x)=\lim_n\lambda^{n}\PP(D^*_n=k_n)$; \eqref{eq:RT} gives
$\lambda^{n}\PP(D^*_n=k_n)\le Ce^{-c\,k_n/\lambda^{n}}\to Ce^{-cx}$, so $w_j(x)\le C_2e^{-c_2x}$ for
$x\ge1$.
\end{proof}

\begin{lemma}[Uniform local limit theorem]\label{lem:llt}
There are $\varepsilon_n\downarrow0$ (depending only on $p$ and $j$) with
$\sup_{k\ge j}\bigl|\lambda^{n}\PP(D^*_n=k)-w_j(k\lambda^{-n})\bigr|\le\varepsilon_n$.
\end{lemma}

\begin{proof}
$D^*_n$ has generating function $[f^{(n)}]^{\,j}$; substituting $t=\theta/\lambda^{n}$ and writing
$x:=k\lambda^{-n}$,
\[
\begin{aligned}
  \lambda^{n}\PP(D^*_n=k)
  &=\frac{\lambda^{n}}{2\pi}\int_{-\pi}^{\pi}f^{(n)}(e^{it})^{j}e^{-ikt}\,dt\\
  &=\frac1{2\pi}\int_{-\pi\lambda^{n}}^{\pi\lambda^{n}}\psi_n(\theta)^{j}e^{-ix\theta}\,d\theta,
\end{aligned}
\]
while Fourier inversion gives $w_j(x)=\frac1{2\pi}\int_\R\widehat w(\theta)^{j}e^{-ix\theta}\,d\theta$.
With this $x$,
\[
\begin{aligned}
  \bigl|\lambda^{n}\PP(D^*_n=k)-w_j(x)\bigr|
  &\le\frac1{2\pi}\int_{\abs\theta\le\pi\lambda^{n}}
  \bigl|\psi_n(\theta)^{j}-\widehat w(\theta)^{j}\bigr|\,d\theta\\
  &\quad+\frac1{2\pi}\int_{\abs\theta>\pi\lambda^{n}}\abs{\widehat w(\theta)}^{j}\,d\theta
  =:\varepsilon_n,
\end{aligned}
\]
independent of $k$. The first integrand is dominated by $2C_1^{j}(1+\abs\theta)^{-j\kappa}\in L^1$
(as $j\kappa>1$) and $\to0$ pointwise by \eqref{eq:cfbound}; the second is an $L^1$-tail. Hence
$\varepsilon_n\to0$, and replacing $\varepsilon_n$ by $\sup_{r\ge n}\varepsilon_r$ makes it
nonincreasing.
\end{proof}

\noindent(The exponential-tail step of \Cref{lem:cf} uses only \Cref{lem:llt}, whose proof is
independent of it; there is no circularity.)

\section{The sharp asymptotic}\label{sec:sharp}

With the local analysis of \Cref{sec:local,sec:ground} in hand, we now construct the periodic factor
$\Psi_j$ and prove the sharp asymptotic \eqref{eq:sharpmain} of \Cref{thm:sharpmain}, and then
compute its Fourier coefficients. The mechanism is to periodize the density profile $w_j$ against the
geometric weights of the mixture \eqref{eq:mix}: summed over generations $n$, the local terms
organize, on the scale $\log_\lambda k$, into a single continuous function of period $1$.

Write $\eta(y):=y^{\rho+1}w_j(y)$ ($y>0$); by \Cref{lem:cf}, $\eta$ is continuous with
$\eta(y)=O(y^{\rho+1})$ as $y\downarrow0$ and $\eta(y)\le C y^{\rho+1}e^{-c_2 y}$ for $y\ge1$, so
$\eta$ is bounded and $\int_0^\infty\eta(y)\tfrac{dy}{y}=\int_0^\infty w_j(y)y^{\rho}dy
=\EE[V_j^{\rho}]<\infty$ (as $\rho<2$ and $V_j$ has a finite second moment).

\begin{proposition}[Periodized profile]\label{prop:phi}
The series $\Phi(u):=\sum_{\ell\in\Z}\eta\bigl(\lambda^{\,u-\ell}\bigr)$ converges locally uniformly
on $\R$, and uniformly on $[0,1]$, to a continuous, strictly positive function of period $1$. Moreover, with
$N=\lfloor\log_\lambda x\rfloor$ and $\{\,\cdot\,\}$ denoting the fractional part,
\begin{equation}\label{eq:Hphi}
\begin{aligned}
  x^{\rho+1}\sum_{n\ge0}\lambda^{-(\rho+1)n}w_j(x\lambda^{-n})
  &=\Phi(\log_\lambda x)-\sum_{\ell<-N}\eta(\lambda^{\{\log_\lambda x\}-\ell})\\
  &=\Phi(\log_\lambda x)+o(1)\qquad(x\to\infty),
\end{aligned}
\end{equation}
the error being uniform in $\{\log_\lambda x\}$.
\end{proposition}

\begin{proof}
For $\ell\to+\infty$, $\lambda^{u-\ell}\downarrow0$ and $\eta(\lambda^{u-\ell})=O(\lambda^{(\rho+1)(u-\ell)})$
is summable; for $\ell\to-\infty$, $\lambda^{u-\ell}\to\infty$ and
$\eta(\lambda^{u-\ell})\le C\lambda^{(\rho+1)(u-\ell)}e^{-c_2\lambda^{u-\ell}}$ is summable, both
uniformly for $u$ in bounded sets. Hence the series converges locally uniformly; each term is continuous,
so $\Phi$ is continuous, and $\Phi(u+1)=\Phi(u)$ by reindexing. Every term is $>0$ (as $w_j>0$ on
$(0,\infty)$), so $\Phi>0$. Writing $\log_\lambda x=N+u$, $u=\{\log_\lambda x\}$, the left side of
\eqref{eq:Hphi} equals $\sum_{n\ge0}\eta(\lambda^{N+u-n})=\sum_{\ell\ge-N}\eta(\lambda^{u-\ell})$
(with $\ell=n-N$), i.e.\ $\Phi(u)$ minus the tail $\sum_{\ell<-N}\eta(\lambda^{u-\ell})$; that tail is
$\le\sum_{r\ge1}C\lambda^{(\rho+1)(u+N+r)}e^{-c_2\lambda^{u+N+r}}\to0$ as $N\to\infty$, uniformly
in $u$.
\end{proof}

Define
\begin{equation}\label{eq:Psidef}
  \Psi_j(u):=\frac{2p}{a}\,\Phi(u),
\end{equation}
which by \Cref{prop:phi} is continuous, strictly positive and of period $1$. This is the function
of \Cref{thm:sharpmain}, whose sharp asymptotic \eqref{eq:sharpmain} we now prove.

\begin{proof}[Proof of \Cref{thm:sharpmain}]
Let $H(x):=\sum_{n\ge0}\lambda^{-(\rho+1)n}w_j(x\lambda^{-n})$, so \Cref{prop:phi} gives
$H(x)=x^{-\rho-1}\Phi(\log_\lambda x)+o(x^{-\rho-1})$. It remains to prove
$p_k=\tfrac{2p}{a}H(k)+o(k^{-1-\rho})$. By \eqref{eq:mix} and $a=\lambda^{\rho}$,
$p_k-\tfrac{2p}{a}H(k)=\tfrac{2p}{a}R(k)$ with
$R(k)=\sum_{n\ge0}\lambda^{-(\rho+1)n}\bigl[\lambda^{n}\PP(D^*_n=k)-w_j(k\lambda^{-n})\bigr]$. We bound
$\abs{R(k)}\le\sum_{n\ge0}\lambda^{-(\rho+1)n}\bigl|\lambda^{n}\PP(D^*_n=k)-w_j(k\lambda^{-n})\bigr|$; fix
$L\in\N$ and split at $\abs{n-N_k}\le L$ versus $>L$, where $N_k=\lfloor\log_\lambda k\rfloor$.

\emph{Window $\abs{n-N_k}\le L$.} \Cref{lem:llt} bounds each bracket by
$\varepsilon_n\le\varepsilon_{N_k-L}$, so the window contributes at most
\[
  \varepsilon_{N_k-L}\sum_{\abs{n-N_k}\le L}\lambda^{-(\rho+1)n}
  \le C_L\,\varepsilon_{N_k-L}\,k^{-1-\rho}=o(k^{-1-\rho}).
\]

\emph{Left tail $n<N_k-L$.} Put $r=N_k-n>L$, so $k/\lambda^{n}\ge\lambda^{r}$. Both terms are
exponentially small: \eqref{eq:RT} gives $\lambda^{n}\PP(D^*_n=k)\le Ce^{-c\lambda^{r}}$, and
\Cref{lem:cf} gives $w_j(k\lambda^{-n})\le C e^{-c_2\lambda^{r}}$. Hence
\[
\begin{aligned}
  \sum_{n<N_k-L}\lambda^{-(\rho+1)n}\bigl|\lambda^{n}\PP(D^*_n=k)-w_j(k\lambda^{-n})\bigr|
  &\le C k^{-1-\rho}\sum_{r>L}\lambda^{(\rho+1)r}e^{-c'\lambda^{r}}\\
  &=\eta_1(L)\,k^{-1-\rho}
\end{aligned}
\]
with $\eta_1(L)\to0$.

\emph{Right tail $n>N_k+L$.} Here $k/\lambda^{n}<\lambda^{-L}$; both terms are merely bounded
($\lambda^{n}\PP(D^*_n=k)\le C$ by \eqref{eq:RT}, $w_j\le\norm{w_j}_\infty$), so
$\sum_{n>N_k+L}\lambda^{-(\rho+1)n}\le C\lambda^{-(\rho+1)(N_k+L)}\le C\lambda^{-(\rho+1)L}k^{-1-\rho}
=\eta_2(L)k^{-1-\rho}$ with $\eta_2(L)\to0$.

Letting $k\to\infty$ then $L\to\infty$ gives $R(k)=o(k^{-1-\rho})$, so
$p_k=\tfrac{2p}{a}H(k)+o(k^{-1-\rho})=k^{-1-\rho}\tfrac{2p}{a}\Phi(\log_\lambda k)+o(k^{-1-\rho})$.
Continuity, positivity and periodicity of $\Psi_j=\tfrac{2p}{a}\Phi$ are those of $\Phi$
(\Cref{prop:phi}).
\end{proof}

The Fourier coefficients of $\Psi_j$ are computed directly from the periodization, with no contour
shift.

\begin{proposition}[Fourier coefficients]\label{prop:fourier}
Let $\chi:=2\pi/\log\lambda$, $s_m:=\rho+1+i\chi m$, and
$\Mel_j(s):=\int_0^\infty w_j(x)x^{s-1}dx=\EE[V_j^{\,s-1}]$, finite and analytic for
$\operatorname{Re}s>1$. Then the continuous period-$1$ function $\Psi_j$ has Fourier coefficients
\begin{equation}\label{eq:fourier}
  \widehat\Psi_j(m):=\int_0^1\Psi_j(u)e^{2\pi i m u}\,du=\frac{2p}{a\log\lambda}\,\Mel_j(s_m).
\end{equation}
In particular $\widehat\Psi_j(0)=\overline\Psi_j=\frac{2p}{a\log\lambda}\EE[V_j^{\,\rho}]>0$.
\end{proposition}

\begin{proof}
By \Cref{prop:phi} and $\Psi_j=\tfrac{2p}{a}\Phi$, uniform convergence of
$\sum_\ell\eta(\lambda^{u-\ell})$ permits termwise integration:
$\widehat\Psi_j(m)=\tfrac{2p}{a}\sum_{\ell\in\Z}\int_0^1\eta(\lambda^{u-\ell})e^{2\pi imu}\,du$.
Substituting $v=u-\ell$ (as $u$ runs over $[0,1)$ and $\ell$ over $\Z$, $v$ runs over $\R$) and
using $e^{2\pi im(v+\ell)}=e^{2\pi imv}$,
\[
\begin{aligned}
  \widehat\Psi_j(m)
  &=\frac{2p}{a}\int_\R\eta(\lambda^{v})e^{2\pi imv}\,dv
  =\frac{2p}{a\log\lambda}\int_0^\infty w_j(x)\,x^{\rho+i\chi m}\,dx\\
  &=\frac{2p}{a\log\lambda}\,\Mel_j(s_m),
\end{aligned}
\]
the middle step setting $x=\lambda^{v}$. Analyticity of $\Mel_j$ for $\operatorname{Re}s>1$ is that
of the moments of $V_j$; the interchange is justified by
$\int_\R\abs{\eta(\lambda^v)}\,dv=\tfrac1{\log\lambda}\EE[V_j^{\rho}]<\infty$.
\end{proof}

\begin{remark}[Koenigs/de Bruijn origin of the period]\label{rem:koenigs}
The period-$1$ (in $\log_\lambda k$) structure is the shadow of the discrete scale $\lambda^{n}$, the
discrete-scale-invariance origin of log-periodicity familiar from asymptotic analysis \cite{deBruijn}.
At the repelling fixed point $s=1$ of $f$ (multiplier $\lambda$),
$Q(\theta)=\lim_nf^{(n)}(e^{-\theta\lambda^{-n}})=\EE e^{-\theta W}$ solves $Q(\lambda\theta)=f(Q(\theta))$,
conjugating $f$ to $\theta\mapsto\lambda\theta$; integer iteration samples this on the geometric
lattice $\{\lambda^{n}\}$, whence the $\log_\lambda$-period. Equivalently, in $t=1-x$ the map
$T(t)=1-f(1-t)=\lambda t-pt^2$ fixes $0$ with multiplier $\lambda$; its Koenigs coordinate \cite{KCG90}
$\kappa_\ast$---the normalized solution of $\kappa_\ast(T(t))=\lambda\kappa_\ast(t)$ ($\kappa_\ast(0)=0$,
$\kappa_\ast'(0)=1$), obtained for $0\le t<1$ as $\kappa_\ast(t)=\lim_n\lambda^{n}b^{(n)}(t)$ via the contracting inverse
branch $b(t)=\tfrac{\lambda-\sqrt{\lambda^2-4pt}}{2p}$ of $T$ (the forward limit
$\lambda^{n}(1-f^{(n)}(1-t))$ diverges, since $1-f^{(n)}(1-t)=T^{(n)}(t)\uparrow1$)---formally linearizes
the singularity of $g$ at $x=1$. This suggests a singular part
$\sim C\,\kappa_\ast(1-x)^{\rho}\,\Pi(\log_\lambda\kappa_\ast(1-x))$ that, under the Flajolet-Odlyzko
transfer theorem \cite{FO}, would give \eqref{eq:sharpmain}; we do not carry this out---the required
$\Delta$-domain continuation and transfer hypotheses are not verified here---as \Cref{thm:sharpmain} is
established probabilistically in \Cref{sec:sharp}.
\end{remark}

\section{Non-constancy of \texorpdfstring{$\Psi_j$}{Psi\_j}}\label{sec:const}

With $\Psi_j$ now in hand, it remains to decide when its oscillation is genuine, i.e.\ when $\Psi_j$
is not constant. By the Fourier formula \eqref{eq:fourier} this is exactly a question about the
vanishing of the Mellin values $\Mel_j(s_m)$ away from the mean $m=0$. We first record this
equivalence, then, for the two network models, reduce it to a single-particle non-tiling problem for
an explicit positive kernel, and settle it for $p$ near $1$ and for all $p$ outside a discrete set
--- thereby proving \Cref{cor:osc}.

\begin{corollary}\label{cor:const}
$\Psi_j$ is constant if and only if $\Mel_j(s_m)=0$ for every $m\in\Z\setminus\{0\}$.
\end{corollary}

\begin{proof}
$\Psi_j$ is continuous and period-$1$; by \eqref{eq:fourier} its Fourier coefficients vanish for
all $m\ne0$ iff all $\Mel_j(s_m)=0$ ($m\ne0$), and a continuous function with only the zeroth
Fourier coefficient nonzero is constant (Fourier uniqueness).
\end{proof}

For the two network models the criterion collapses to a single-particle one.

\begin{proposition}[Single-particle reduction, $j\in\{1,2\}$]\label{prop:reduce}
Let $z_m:=\rho+i\chi m$. Then $\Mel_1(s_m)=\EE[W^{z_m}]$ and $\Mel_2(s_m)=3\,\EE[W^{z_m}]$. Hence
for $j\in\{1,2\}$, $\Psi_j$ is constant iff $\EE[W^{z_m}]=0$ for all $m\ne0$, i.e.\ iff under the
tilted law $\propto W^{\rho}\,d\PP$ the fractional part $\{\log_\lambda W\}$ is uniform on $[0,1)$.
\end{proposition}

\begin{proof}
$\Mel_1(s_m)=\EE[W^{z_m}]$ is immediate, as $V_1=W$. For $j=2$, the distributional fixed point
$W\eqd\lambda^{-1}\sum_{i=1}^{\xi}W_i$ ($\xi=1+\Ber(p)$) gives, for $\operatorname{Re}u>0$,
\[
  M_W(u):=\EE[W^{u}]=\lambda^{-u}\bigl(qM_W(u)+pN(u)\bigr),\qquad N(u):=\EE[(W_1+W_2)^{u}],
\]
hence $N(u)=\tfrac{\lambda^{u}-q}{p}M_W(u)$. At $u=z_m$, $\lambda^{i\chi m}=e^{2\pi i m}=1$, so
$\lambda^{u}=\lambda^{\rho}=a$ and $\lambda^{u}-q=a-q=3p$; since $\Mel_2(s_m)=N(z_m)$, the claim
follows. The tilting statement is $\EE[W^{z_m}]=\EE[W^{\rho}e^{2\pi i m\log_\lambda W}]$.
\end{proof}

Thus, for $j\in\{1,2\}$, $\Psi_j$ is constant precisely under a multiplicative-lattice degeneracy
of $W$: the $W^{\rho}$-tilted law of $\{\log_\lambda W\}$ is exactly uniform. This is reminiscent
of the Karlin-McGregor/Dubuc \emph{near-constancy} phenomenon \cite{Dubuc82,BB91} and of embeddability of the
discrete process into a continuous-time branching process (which fails here: an embedding would
require a pgf $h$ with $h\circ h=f$, and $f(x)=px^2+qx$ has no such compositional square root---for a
polynomial $h$, $\deg(h\circ h)=(\deg h)^2\ne2$, while an infinite-support $h$ forces $h\circ h$ to
have infinite support by nonnegativity of coefficients \cite{KM}). We caution, however, that the classical criterion governs a
\emph{different} Mellin line, namely the Schr\"oder/left-tail Harris function
\cite{Harris63,CG13}, not the $\rho$-line
values $\EE[W^{z_m}]$ at issue here, so it does not by itself decide the behaviour at the isolated
exceptional points of $\mathcal Z_1$. Instead we reduce the question to a sharp non-tiling problem,
and prove non-degeneracy near $p=1$.

\begin{proposition}[Positive-kernel reduction and a rigid shift ratio]\label{prop:kernel}
Let $\varphi(t)=\EE e^{-tW}$ and $y(t)=1-\varphi(t)$, so the Laplace fixed point reads
\begin{equation}\label{eq:ypoinc}
  y(\lambda t)=\lambda\,y(t)-p\,y(t)^2=y(t)\bigl(\lambda-p\,y(t)\bigr).
\end{equation}
Then, for every $m\in\Z$, with $z_m:=\rho+i\chi m$ and $C(z):=1/\Gamma(-z)$,
\begin{equation}\label{eq:ysquare}
  \EE[W^{z_m}]=C(z_m)\int_0^\infty y(t)^2\,t^{-z_m-1}\,dt .
\end{equation}
Writing $t=\lambda^{v}$ and $H_p(v):=\lambda^{-\rho v}\,y(\lambda^{v})^{2}\ge0$, \eqref{eq:ysquare}
reads $\EE[W^{z_m}]=C(z_m)\log\lambda\,\widehat H_p(m)$, $\widehat H_p(m)=\int_\R H_p(v)e^{-2\pi imv}dv$;
since $C(z_m)\ne0$, \Cref{cor:const} and \Cref{prop:reduce} show that, for $j\in\{1,2\}$, $\Psi_j$
is constant iff the periodization
$\mathcal P_p(u):=\sum_{n\in\Z}H_p(u+n)$ is constant. Moreover the unit-shift ratio of $H_p$,
\begin{equation}\label{eq:ratio}
  \frac{H_p(v+1)}{H_p(v)}=\lambda^{2-\rho}\Bigl(1-\tfrac{p}{\lambda}\,y(\lambda^{v})\Bigr)^{2}
\end{equation}
is strictly decreasing in $v$, from $\lambda^{2-\rho}>1$ (as $v\to-\infty$) to $\lambda^{-\rho}<1$
(as $v\to+\infty$).
\end{proposition}

\begin{proof}
Equation \eqref{eq:ypoinc} is $\varphi(\lambda t)=q\varphi+p\varphi^2$ rewritten via $y=1-\varphi$
and $q+2p=\lambda$. For $1<\operatorname{Re}z<2$ the twice-subtracted Frullani identity \cite{FGD95} gives
$\EE W^{z}=C(z)\int_0^\infty(\varphi(t)-1+t)t^{-z-1}dt$ (valid as $\EE W=1$); put
$A(t):=\varphi(t)-1+t=t-y(t)$. By \eqref{eq:ypoinc}, $A(\lambda t)=\lambda A(t)+p\,y(t)^2$, so with
$I_A(z)=\int_0^\infty A(t)t^{-z-1}dt$ and $I_y(z)=\int_0^\infty y(t)^2t^{-z-1}dt$, the substitution
$t\mapsto\lambda t$ yields $I_A(z)=\lambda^{1-z}I_A(z)+p\lambda^{-z}I_y(z)$, i.e.\
$I_A(z)=\frac{p\lambda^{-z}}{1-\lambda^{1-z}}I_y(z)$. At $z=z_m$, $\lambda^{z_m}=\lambda^{\rho}=a$,
so $p\lambda^{-z_m}=p/a$ and $1-\lambda^{1-z_m}=1-\lambda/a=(a-\lambda)/a=p/a$; the prefactor is
$1$, hence $I_A(z_m)=I_y(z_m)$ and \eqref{eq:ysquare} follows. The substitution $t=\lambda^{v}$
turns $I_y(z_m)$ into $\log\lambda\int_\R\lambda^{-\rho v}y(\lambda^v)^2e^{-2\pi imv}dv=\log\lambda\,\widehat H_p(m)$.
Since $y(t)\sim t$ as $t\downarrow0$ and $y(t)\uparrow1$ as $t\to\infty$, the kernel
$H_p(v)=\lambda^{-\rho v}y(\lambda^v)^2$ satisfies $H_p(v)=O(\lambda^{(2-\rho)v})$ as $v\to-\infty$ and
$H_p(v)=O(\lambda^{-\rho v})$ as $v\to+\infty$; as $2-\rho,\rho>0$, the integer-shift series
$\mathcal P_p(u)=\sum_{n\in\Z}H_p(u+n)$ converges uniformly on $[0,1]$, so
$\mathcal P_p\in C(\mathbb T)\subset L^1(\mathbb T)$ with $\widehat{\mathcal P_p}(m)=\widehat H_p(m)$. By
Fourier uniqueness on $\mathbb T$, all $\widehat H_p(m)=0$ ($m\ne0$) iff $\mathcal P_p$ is constant. Finally \eqref{eq:ypoinc} gives
$y(\lambda^{v+1})/y(\lambda^{v})=\lambda-p\,y(\lambda^v)$, so
$H_p(v+1)/H_p(v)=\lambda^{-\rho}(\lambda-p\,y(\lambda^v))^2=\lambda^{2-\rho}(1-\tfrac p\lambda y(\lambda^v))^2$;
as $y$ increases strictly from $0$ to $1$, the factor $(1-\tfrac p\lambda y(\lambda^v))$ decreases
from $1$ to $1-p/\lambda=1/\lambda$, giving the stated monotone range
$\lambda^{2-\rho}\downarrow\lambda^{-\rho}$.
\end{proof}

\begin{remark}
The remaining open case is thus a concrete \emph{non-tiling} problem: show that the explicit
positive kernel $H_p(v)=\lambda^{-\rho v}(1-\varphi(\lambda^{v}))^{2}$, whose rigid, strictly
decreasing unit-shift ratio \eqref{eq:ratio} crosses $1$ exactly once, cannot form an exact
partition of unity under integer shifts. Ratio-monotonicity alone is not sufficient in general
(compactly supported log-concave kernels, e.g.\ $B$-splines \cite{deBoor}, do tile), so a proof must use the
full-support analyticity of $H_p$ or the Schr\"oder structure of $y$.
\end{remark}

\begin{proposition}[Non-degeneracy near $p=1$]\label{prop:near1}
For every $j\ge1$,
\[
  \lim_{p\uparrow1}\abs{\widehat\Psi_j(1)}=\frac{2}{3\log 2}\,j^{\rho_*}\ne0,\qquad
  \rho_*=\log3/\log2 .
\]
Hence there is $p_0=p_0(j)<1$ such that $\Mel_j(s_1)\ne0$ and thus $\Psi_j$ is non-constant for every
$p\in(p_0,1)$.
\end{proposition}

\begin{proof}
By \Cref{prop:fourier}, $\widehat\Psi_j(1)=\tfrac{2p}{a\log\lambda}\EE[V_j^{\rho+i\chi}]$ with
$\rho=\rho(p)$, $\chi=\chi(p)$. As $p\uparrow1$, $\EE W=1$ and, since the offspring variance is
$\operatorname{Var}(\xi)=pq$, $\operatorname{Var}(W)=\tfrac{pq}{\lambda(\lambda-1)}=\tfrac{q}{\lambda}=\tfrac{1-p}{1+p}\to0$, so $W\to1$
in probability, whence $V_j=W_1+\cdots+W_j\to j$ in probability. The family
$\{V_j^{\rho(p)}\}_{p\in[1/2,1)}$ is uniformly integrable: as $\rho$ is strictly decreasing
(\Cref{sec:results}), $\bar\rho:=\sup_{p\in[1/2,1)}\rho(p)=\rho(1/2)=\log2/\log\tfrac32<2$, so we may
fix $\delta>0$ with $\bar\rho(1+\delta)\le2$; then $V_j^{\rho(p)(1+\delta)}\le1+V_j^{2}$, while
$\EE V_j^{2}=j\,q/\lambda+j^{2}\le j+j^{2}$ (as $\operatorname{Var}W=q/\lambda\le1$), so
$\sup_{p\in[1/2,1)}\EE[V_j^{\rho(p)(1+\delta)}]<\infty$. Since $\rho(p)\to\rho_*$ and
$\chi(p)\to\chi_*:=2\pi/\log2$, both finite, $V_j^{\rho+i\chi}\to j^{\rho_*+i\chi_*}$ in probability, hence in
$L^1$ by uniform integrability; so $\EE[V_j^{\rho+i\chi}]\to j^{\rho_*+i\chi_*}$, of modulus
$j^{\rho_*}$. Finally $\tfrac{2p}{a\log\lambda}\to\tfrac{2}{3\log2}$, giving the claim.
\end{proof}

For $j=1$ the limit is $\tfrac{2}{3\log2}$, and for $j=2$ it is
$\tfrac{2}{3\log2}\cdot2^{\rho_*}=\tfrac{2}{\log2}$ (as $2^{\log3/\log2}=3$).

\begin{proposition}[Real-analytic dependence on $p$]\label{prop:analytic}
For each fixed $m\in\Z$, the map $p\mapsto F_m(p):=\EE_p\bigl[W_p^{\,z_m(p)}\bigr]$, where
$z_m(p)=\rho(p)+i\chi(p)m$, is real-analytic on $(0,1)$.
\end{proposition}

\begin{proof}
Fix $p_0\in(0,1)$ and a complex neighbourhood $U\ni p_0$ on which $\lambda(p)=1+p$, $a(p)=1+2p$ and
their logarithms are single-valued, with $|\lambda|>1$, $|a|>1$, $|a|<|\lambda|^2$ and $|1-p|<1$
(all strict at $p_0$). We continue the \emph{deterministic} Poincar\'e function \cite{KCG90}, not the
probabilistic Laplace transform. The series $Y_p(t)=\sum_{r\ge1}b_r(p)t^{r}$, $b_1\equiv1$, solving
$Y_p(\lambda t)=\lambda Y_p(t)-pY_p(t)^2$ has coefficients
$(\lambda^{r}-\lambda)b_r=-p\sum_{i+i'=r}b_ib_{i'}$ ($r\ge2$), each holomorphic in $p$ on $U$ (the
denominators $\lambda^{r}-\lambda$ do not vanish, as $\abs\lambda>1$). Since
$\abs{\lambda^{r}-\lambda}\ge(1-\abs\lambda^{-1})\abs\lambda^{r}\ge(1-\abs\lambda^{-1})\abs\lambda^{2}$
for $r\ge2$, the recursion gives $\abs{b_r}\le A\sum_{i+i'=r}\abs{b_i}\abs{b_{i'}}$ with
$A:=\sup_{p\in\overline U}\abs p/[(1-\abs\lambda^{-1})\abs\lambda^{2}]<\infty$, so the Catalan majorant
$\mathcal B(t)\preceq t+A\,\mathcal B(t)^2$ \cite{FlS09} bounds $\sum_r\abs{b_r}\,t^r$ on $\abs t<1/(4A)$.
Hence $Y_p(t)$ is holomorphic in $(p,t)$ on $U\times \mathbb D(0,\varepsilon)$ for a common
$\varepsilon>0$; for real $p$,
$Y_p(t)=y_p(t):=1-\varphi_p(t)$ on $\abs t<\varepsilon$: by \Cref{lem:expmom} and Fatou
$\EE e^{\theta W}\le M<\infty$, so $\varphi_p(t)=\EE e^{-tW}$ is analytic near $0$; $y_p=1-\varphi_p$
satisfies $y_p(\lambda t)=\lambda y_p(t)-py_p(t)^2$ (\eqref{eq:ypoinc}) with $y_p(0)=0$ and
$y_p'(0)=\EE W=1$, which is exactly the normalization and recursion defining $Y_p$, so $Y_p=y_p$ near
$0$, hence, as identical power series, on all of $\abs t<\varepsilon$. Let $g_p(z)=\lambda z-pz^2$, with fixed points $0$ (multiplier $g_p'(0)=\lambda$) and $1$
(multiplier $g_p'(1)=\lambda-2p=q$), and let $b_p(z)=\tfrac{\lambda-\sqrt{\lambda^2-4pz}}{2p}$ be its
inverse branch fixing $0$ ($b_p(0)=0$, $b_p'(0)=\lambda^{-1}$). As $\lambda(p_0)^2/(4p_0)=(1+p_0)^2/(4p_0)>1$,
fix a real radius $R$ with $1<R<(1+p_0)^2/(4p_0)$ and shrink $U$ to a relatively compact disk with
$\inf_{p\in\overline U}\abs{\lambda(p)^2/(4p)}>R$; then $\lambda^2-4pz\ne0$ on $U\times\{\abs z<R\}$, and,
taking the square root equal to $\lambda$ at $z=0$, $b_p(z)$ is jointly holomorphic on the
simply-connected disk $\{\abs z<R\}$, which contains $(0,1)$ and, after a further shrinking of $U$, the
whole backward orbit constructed below.

\emph{A base curve on a fundamental domain.} The disk $\mathbb D(0,\varepsilon)$ need not contain
$\{\lambda^{u}:u\in[0,1]\}$, so we reach that segment by iteration. Fix an integer $K$ with
$\sup_{u\in[0,1],\,p\in\overline U}\abs{\lambda^{\,u-K}}<\varepsilon$ (possible as $\abs\lambda>1$ on $\overline U$) and put
\[
  \Upsilon_p(u):=g_p^{\,K}\bigl(Y_p(\lambda^{\,u-K})\bigr)\qquad(u\in[0,1]),
\]
holomorphic in $p\in U$ and jointly continuous in $(p,u)$ (as $\lambda^{u-K}\in\mathbb D(0,\varepsilon)$
and $g_p$ is a polynomial); for real $p$, the functional equation $g_p(y_p(t))=y_p(\lambda t)$ gives
$\Upsilon_p(u)=y_p(\lambda^{u})$. Set
\[
  Y_{p,n}:=g_p^{\,n}\!\circ\Upsilon_p\ (n\ge0),\qquad Y_{p,-k}:=b_p^{\,k}\!\circ\Upsilon_p\ (k\ge1),
\]
so that, for real $p$, $Y_{p,n}(u)=y_p(\lambda^{\,u+n})$ for every $n\in\Z$.

\emph{Uniform basin estimates.} At $p_0$ the curve $\Upsilon_{p_0}([0,1])=y_{p_0}([1,\lambda])$ is a
compact subset of $(0,1)$, on which $g_{p_0}$ increases $(0,1)$ into itself with
$g_{p_0}(z)-z=pz(1-z)>0$, and $b_{p_0}$ likewise with $b_{p_0}(z)<z$; hence
$g_{p_0}^{\,n}\!\circ\Upsilon_{p_0}\to1$ and $b_{p_0}^{\,k}\!\circ\Upsilon_{p_0}\to0$ uniformly on
$[0,1]$. Fix a constant $\theta$ with $\sup_{p\in\overline U}\abs{\lambda(p)}^{-1}<\theta<\inf_{p\in\overline U}\abs{a(p)}^{-1/2}$
(a nonempty range once $U$ is small, as $\abs{a(p_0)}<\abs{\lambda(p_0)}^2$), and closed disks
$\overline{\mathbb D}(1,r_1)$, $\overline{\mathbb D}(0,r_0)\subset\{\abs z<R\}$ on which,
after shrinking $U$, $g_p$ is a contraction (factor $<1$, as $\abs q<1$) and $b_p$ a contraction with
factor $\le\theta$ (as $\abs{b_p'(0)}=\abs{\lambda}^{-1}<\theta$), uniformly in $p\in U$; by
the uniform convergence at $p_0$ and continuity in $p$ there are $N_1,N_0$ and a smaller $U$ with
$g_p^{\,N_1}\Upsilon_p([0,1])\subset\overline{\mathbb D}(1,r_1)$ and
$b_p^{\,N_0}\Upsilon_p([0,1])\subset\overline{\mathbb D}(0,r_0)$ for all $p\in U$, so all subsequent backward
iterates ($k\ge N_0$) stay in $\overline{\mathbb D}(0,r_0)$, the whole orbit remaining in $\{\abs z<R\}$, where $b_p$ is holomorphic. Consequently
$\abs{Y_{p,n}(u)}\le M$ for all $n\ge0$ (the finitely many $n<N_1$ bounded on the compact
$\overline U\times[0,1]$) and $\abs{Y_{p,-k}(u)}\le C\theta^{\,k}$, both uniformly in $u\in[0,1]$, $p\in U$; whence
$\abs{a^{-(u+n)}Y_{p,n}^2}\le M'\abs a^{-n}$ and $\abs{a^{-(u-k)}Y_{p,-k}^2}\le C'(\abs a\theta^{2})^{k}$
with $\abs a>1$ and $\abs a\theta^{2}<1$ (by $\theta<\inf_{p\in\overline U}\abs{a(p)}^{-1/2}$). Both tails
therefore converge locally uniformly in $p\in U$. Thus
\[
  A_m(p):=\int_0^1 e^{-2\pi imu}\sum_{n\in\Z}a(p)^{-(u+n)}Y_{p,n}(u)^2\,du
\]
is holomorphic on $U$ (each summand is continuous in $(p,u)$ and holomorphic in $p$; the locally
uniform convergence keeps the sum holomorphic in $p$, and Morera's theorem with Fubini's theorem show
that integration over the compact $[0,1]$ preserves holomorphy). For real $p$, $a^{-(u+n)}Y_{p,n}(u)^2=H_p(u+n)$, so $A_m(p)=\widehat H_p(m)$.
By \Cref{prop:kernel}, $F_m(p)=C(z_m(p))\,\log\lambda(p)\,A_m(p)$ with $C(z_m)\ne0$ and the prefactor
holomorphic, so $F_m$ is real-analytic near $p_0$.
\end{proof}

\begin{corollary}[Non-degeneracy off a discrete set]\label{cor:genericp}
There is a discrete set $\mathcal Z_1\subset(0,1)$ such that, for both $j\in\{1,2\}$, $\Psi_j$ is
non-constant for every $p\in(0,1)\setminus \mathcal Z_1$.
\end{corollary}

\begin{proof}
$F_1(p)=\EE_p[W^{z_1}]$ is real-analytic (\Cref{prop:analytic}) and $F_1(p)\to1$ as $p\uparrow1$
(as $W\to1$ in probability with $\{W^{\rho}\}$ uniformly integrable), so $F_1\not\equiv0$ and its
zero set $\mathcal Z_1=\{p:F_1(p)=0\}$ is discrete. For $p\notin \mathcal Z_1$ and $j\in\{1,2\}$, \Cref{prop:reduce}
gives $\Mel_j(s_1)=c_j F_1(p)\ne0$ ($c_1=1$, $c_2=3$), so $\Psi_j$ is non-constant.
\end{proof}

Combining \Cref{cor:genericp,prop:near1} gives \Cref{cor:osc}: for $j\in\{1,2\}$, $\Psi_j$ is
non-constant for every $p\in(0,1)\setminus \mathcal Z_1$, and, for every $j\ge1$, on a whole
neighbourhood of $p=1$. Whenever $\Psi_j$ is non-constant, $p_k$ is asymptotic to no constant
multiple of $k^{-1-\rho}$: for any $u\in[0,1)$ the integers $k_r=\lceil\lambda^{\,r+u}\rceil$ satisfy
$\{\log_\lambda k_r\}\to u$, so $\{\log_\lambda k\}_{k\ge1}$ is dense in $[0,1)$ and, by
\eqref{eq:sharpmain} and the continuity of $\Psi_j$, the values $p_k\,k^{1+\rho}$ are dense in the
nondegenerate interval $\Psi_j([0,1])$. Settling the behaviour at the isolated exceptional points of
$\mathcal Z_1$ would require some $m\ge2$ with $\Mel_j(s_m)\ne0$, equivalently a proof that the
rigid Schr\"oder kernel $H_p$ of \Cref{prop:kernel} does not tile. This remains open.

\begin{remark}
Consequently, while a genuine equivalence $p_k\sim C\,k^{-1-\rho}$ with constant $C$ holds for no
$p$ near $1$, and (for $j\in\{1,2\}$) for no $p\in(0,1)\setminus\mathcal Z_1$, the logarithmic average
of $k^{1+\rho}p_k$ still converges to the mean $\overline\Psi_j=\widehat\Psi_j(0)$:
\[
  \frac1{\log K}\sum_{k=1}^{K}\frac{k^{1+\rho}p_k}{k}\xrightarrow[K\to\infty]{}\int_0^1\Psi_j(u)\,du=\overline\Psi_j .
\]
Indeed $k^{1+\rho}p_k=\Psi_j(\log_\lambda k)+o(1)$ by \eqref{eq:sharpmain}, and, grouping
$k\in[\lambda^{m},\lambda^{m+1})$ and using the period-$1$ structure and uniform continuity of
$\Psi_j$, the $1/k$-weighted blocks form Riemann sums of $\int_0^1\Psi_j$, while
$\sum_{k\le K}1/k\sim\log K$ normalizes the average and the $o(1)$ term contributes $o(1)$. The correct
pointwise statement remains the oscillatory \eqref{eq:sharpmain}.
\end{remark}

\paragraph{Scope of the results.} Every result above is proved from the local estimates of
\Cref{sec:local} together with \Cref{lem:cf,lem:llt}; the only external inputs are the
almost-sure martingale convergence $\lambda^{-n}D_n\to W$ with $\EE W=1$ (Kesten-Stigum \cite{KS66}),
textbook Fourier analysis (inversion and uniqueness) and, for the strict positivity of $\Psi_j$, the classical
positivity $w>0$ of the Schr\"oder-case density. The existence and Fourier form
\eqref{eq:sharpmain}--\eqref{eq:fourier} of $\Psi_j$ hold for every $j\ge1$, as does its
non-constancy for $p$ near $1$ (\Cref{prop:near1}). The stronger conclusion of \Cref{cor:osc}
asserts non-constancy for all $p$ outside a discrete set $\mathcal Z_1$ and rests on the single-particle
reduction \Cref{prop:reduce} and is therefore established only for the two network models
$j\in\{1,2\}$; the behaviour at the isolated points of $\mathcal Z_1$ is left open.

\section*{Acknowledgements}
The authors thank Yuliang Chen of the School of the Gifted Young, University of Science and
Technology of China, for using the Eureka system to produce a preliminary version of the proof.
Eureka is a multi-agent system for resolving mathematical conjectures through human-AI interaction.
This work is supported by the Fundamental and Interdisciplinary Disciplines Breakthrough Plan of the
Ministry of Education of China.

\end{document}